\documentclass[12pt,reqno]{amsart}

\newtheorem{theorem}{Theorem}[section]
\newtheorem{proposition}[theorem]{Proposition}

\newtheorem{lemma}[theorem]{Lemma}
\newtheorem{corollary}[theorem]{Corollary}
\newtheorem{remark}[theorem]{Remark}

\numberwithin{equation}{section}

\begin{document}

\title[Yang-Mills-Higgs connections on Calabi-Yau manifolds]{Yang-Mills-Higgs
connections\\[5pt] on Calabi-Yau manifolds}

\author[I. Biswas]{Indranil Biswas}

\address{School of Mathematics, Tata Institute of Fundamental
Research, Homi Bhabha Road, Bombay 400005, India}

\email{indranil@math.tifr.res.in}

\author[U. Bruzzo]{Ugo Bruzzo}

\address{Departamento de Matem\'atica, Universidade Federal de Santa Catarina, 
  Campus Universit\'ario Trindade,
CEP 88.040-900 Florian\'opolis-SC, Brazil; Istituto Nazionale di Fisica Nucleare, Sezione di Trieste, Italy. On leave from Scuola Internazionale Superiore di Studi Avanzati, Trieste, Italy}

\email{bruzzo@sissa.it}

\author[B. Gra\~na Otero]{\\[4pt] Beatriz Gra\~na Otero}

\address{Departamento de Matem\'aticas, Pontificia Universidad Javeriana,
Cra. 7$^{\rm ma}$ N$^{\rm o}$ 40-62, Bogot\'a, Colombia}

\email{bgrana@javeriana.edu.co}

\author[A. Lo Giudice]{Alessio Lo Giudice}

\address{IMECC - UNICAMP, Departamento de Matem\'atica, Rua S\'ergio Buarque de Holanda,
651, Cidade Universit\'aria, 13083-859 Campinas--SP, Brazil}

\email{logiudice@ima.unicamp.br}



\date{}

\begin{abstract}
Let $X$ be a compact connected K\"ahler--Einstein manifold with $c_1(TX)\, \geq\, 0$. If there is a semistable Higgs vector bundle $(E\, ,\theta)$ on $X$ with $\theta\,\not=\,
0$, then we show that $c_1(TX)=0$; any $X$ satisfying this condition is called a
Calabi--Yau manifold, and it admits a Ricci--flat K\"ahler form \cite{Ya}. Let $(E\, ,\theta)$
be a polystable Higgs vector bundle on a compact Ricci--flat K\"ahler manifold $X$.
Let $h$ be an Hermitian structure on $E$ satisfying the Yang--Mills--Higgs equation
for $(E\, ,\theta)$. We prove that $h$ also satisfies the Yang--Mills--Higgs equation
for $(E\, ,0)$. A similar result is proved for Hermitian structures on principal Higgs
bundles on $X$ satisfying the Yang--Mills--Higgs equation.
\end{abstract}

\maketitle

\section{Introduction}\label{sec1}

Let $X$ be a compact connected K\"ahler--Einstein manifold with $c_1(TX)\, \geq\, 0$.
A Higgs vector bundle on $X$ is a holomorphic vector bundle $E$
on $X$ equipped with a holomorphic section $\theta$ of
$\text{End}(E)\bigotimes\Omega_X$ such that $\theta\bigwedge\theta\,=\, 0$.
The definition of semistable and polystable Higgs vector bundles is recalled in Section \ref{sec2}.
We prove that if there is a semistable Higgs vector bundle $(E\, ,\theta)$ on $X$ with
$\theta\,\not=\, 0$, then $c_1(TX)\,=\,0$ (see Proposition \ref{cor1}).

Let $X$ be a compact connected Calabi--Yau manifold, which means that $X$ is a
K\"ahler manifold with $c_1(TX)\,=\, 0$. Fix a Ricci--flat K\"ahler form on
$X$ \cite{Ya}. Let $(E\, ,\theta)$ be a polystable Higgs vector bundle
on $X$. Then there is a Hermitian structure on $E$ that satisfies the Yang--Mills--Higgs
equation for $(E\, ,\theta)$ (this equation is recalled in Section \ref{sec2}). Fix a
Hermitian structure $h$ on $E$ satisfying the Yang--Mills--Higgs
equation for $(E\, ,\theta)$. 

Our main theorem (Theorem \ref{thm1}) says that $h$ also satisfies the Yang--Mills--Higgs
equation for $(E\, ,0)$. 

We give an example to show that if a Hermitian structure
$h_0$ on $E$ satisfies the Yang--Mills--Higgs equation
for $(E\, ,0)$, then $h_0$ does not satisfy the Yang--Mills--Higgs equation for
a general polystable Higgs vector bundle of the
form $(E\, ,\theta)$ (see Remark \ref{rem1}). In Remark
\ref{rem2} we describe how a Yang--Mills--Higgs Hermitian structure for $(E\, ,\theta)$
can be constructed from a Yang--Mills--Higgs Hermitian structure for $(E\, ,0)$.

Theorem \ref{thm1} extends to the more general context of
principal $G$--bundles on $X$ with a Higgs structure,
where $G$ is a connected reductive affine algebraic group defined over $\mathbb C$; this
is carried out in Section \ref{se4}.

\bigskip
\section{Higgs field on a K\"ahler-Einstein manifold}\label{sec2}

We recall that a K\"ahler metric is called \textit{K\"ahler--Einstein} if
its Ricci curvature is a constant real multiple of the K\"ahler form.
Let $X$ be a compact connected K\"ahler manifold admitting a K\"ahler--Einstein
metric. We assume
that $c_1(TX)\, \geq\, 0$; this is equivalent to the condition that the above
mentioned scalar factor is nonnegative. Fix a K\"ahler--Einstein form $\omega$ on $X$.
The cohomology class in
$H^2(X,\, {\mathbb R})$ given by $\omega$ will be denoted by $\widetilde{\omega}$.

Define the \textit{degree} of a torsionfree coherent analytic sheaf $F$ on $X$ to be
$$
\text{degree}(F)\, :=\, (c_1(F)\cup \widetilde{\omega}^{d-1})\cap [X]\, \in\,
\mathbb R\, ,
$$
where $d$ is the complex dimension of $X$. Throughout this paper, stability will be with
respect to this definition of degree.

The holomorphic cotangent bundle of $X$ will be denoted by $\Omega_X$. A
\textit{Higgs field} on a holomorphic vector bundle $E$ on $X$ is a holomorphic
section $\theta$ of ${\rm End}(E)\bigotimes \Omega_X\,=\, (E \bigotimes\Omega_X)\bigotimes
E^*$ such that
\begin{equation}\label{e1}
\theta\bigwedge\theta\,=\, 0\, .
\end{equation}
A \textit{Higgs vector bundle} on $X$ is
a pair of the form $(E\, ,\theta)$, where $E$ is a holomorphic vector bundle on $X$
and $\theta$ is a Higgs field on $E$.

A Higgs vector bundle $(E\, ,\theta)$ is called \textit{stable} (respectively,
\textit{semistable}) if for all nonzero coherent analytic subsheaves $F\, \subset\, E$
with $0\, <\,\text{rank}(F)\, <\, \text{rank}(E)$ and $\theta(F)\, \subseteq\,
F \bigotimes\Omega_X$, we have
$$
\frac{{\rm degree}(F)}{{\rm rank}(F)}\, <\, \frac{{\rm degree}(E)}{{\rm rank}(E)}
\ \ {\rm (respectively,\,}~ \frac{{\rm degree}(F)}{{\rm rank}(F)}\, \leq
\, \frac{{\rm degree}(E)}{{\rm rank}(E)}{\rm )}\, .
$$
A semistable Higgs vector bundle $(E\, ,\theta)$ is called \textit{polystable}
if it is a direct sum of stable Higgs vector bundles.

Let $\Lambda_\omega$ denote the adjoint of multiplication of differential
forms on $X$ by $\omega$. In particular, $\Lambda_\omega$ sends a $(p\, ,q)$--form
on $X$ to a $(p-1\, ,q-1)$--form.
Given a Higgs vector bundle $(E\, ,\theta)$ on $X$, the \textit{Yang--Mills--Higgs}
equation for the Hermitian structures $h$ on $E$ states that
\begin{equation}\label{ymh}
\Lambda_\omega ({\mathcal K}_h +\theta\wedge\theta^*)\,=\, c\sqrt{-1}\cdot \text{Id}_E\, ,
\end{equation}
where ${\mathcal K}_h\, \in\, C^\infty(X,\, {\rm End}(E)\bigotimes{\Omega}^{1,1}_X)$
is the curvature of the Chern connection on $E$ for $h$, the adjoint $\theta^*$ of
$\theta$ is with respect to $h$, and $c$ is a constant
scalar (it lies in $\mathbb R$). A Hermitian structure on $E$
is called Yang--Mills--Higgs for $(E\, ,\theta)$ if it satisfies the
equation in \eqref{ymh}.

\begin{proposition}\label{cor1}
If there is a semistable Higgs bundle $(E\, ,\theta)$ on $X$ such that
$\theta\, \not=\, 0$, then $c_1(TX)\,=\, 0$.
\end{proposition}

\begin{proof}
The Higgs field $\theta$ on $E$ induces a Higgs field on ${\rm End}(E)$, which we
will denote by $\widehat{\theta}$. We recall that for any locally defined holomorphic
sections $s$ of ${\rm End}(E)$,
$$
\widehat{\theta}(s) \,=\, [\theta\, ,s]\, .
$$
Let
\begin{equation}\label{tp}
\theta'\, =\, \widehat{\theta}\otimes \text{Id}_{\Omega_X}.
\end{equation}
This  is a
 Higgs field for ${\rm End}(E)\bigotimes\Omega_X$.  We note that the integrability
condition in \eqref{e1} implies that $\theta'(\theta)\,=\, 0$.

Assume that $(E\, ,\theta)$ is semistable with $\theta\, \not=\, 0$, and also
assume that $c_1(TX)\,\not=\, 0$. Since $(X\, ,\omega)$ is K\"ahler--Einstein
with $c_1(TX)\, \geq\, 0$, the condition $c_1(TX)\,\not=\, 0$ implies that
the anti-canonical line bundle $\bigwedge^d TX$ is positive, so $X$ is a complex
projective manifold. Also, the cohomology class of $\omega$ is a positive multiple
of the ample class $c_1(TX)$.

We shall use the fact that the tensor product of semistable Higgs bundles
on a polarized complex projective manifold, with the
induced Higgs field, is semistable \cite[Cor.~3.8]{Si2}. 
Thus,  $({\rm End}(E),\widehat\theta)$ is semistable. Moreover, since 
$\omega$ is K\"ahler--Einstein, $\Omega_X$ is a polystable vector bundle, in particular
it is semistable. Then $(\Omega_X,0)$ is a semistable Higgs bundle. As a result,
the Higgs bundle $({\rm End}(E)\bigotimes\Omega_X,\theta')$ is semistable.

The homomorphism
$$
{\mathcal O}_X\, \longrightarrow\,\, {\rm End}(E)\otimes\Omega_X\, ,~ \
f\longmapsto\, f\theta
$$
defines a homomorphism of Higgs vector bundles
\begin{equation}\label{e2}
\varphi\, :\, ({\mathcal O}_X\, ,0)\,\longrightarrow\, ({\rm End}(E)
\otimes\Omega_X\, ,\theta')\, .
\end{equation}

As $\theta\,\not=\, 0$,
the homomorphism $\varphi$ in \eqref{e2} is nonzero. Since $({\rm End}(E) \otimes\Omega_X\, ,\theta')$ is semistable,  we have
\begin{equation}\label{e4}
0\,=\, \frac{{\rm degree}({\mathcal O}_X)}{{\rm rank}({\mathcal O}_X)}\,=\,
\frac{{\rm degree}(\varphi({\mathcal O}_X))}{{\rm rank}(\varphi({\mathcal O}_X))}
\, \leq\, \frac{{\rm degree}({\rm End}(E)
\otimes\Omega_X)}{{\rm rank}({\rm End}(E)\otimes\Omega_X)}\,=\,
\frac{{\rm degree}(\Omega_X)}{{\rm rank}(\Omega_X)}\, ;
\end{equation}
the last equality follows from the fact that $c_1({\rm End}(E))\,=\, 0$. Therefore,
\begin{equation}\label{e3}
{\rm degree}(\Omega_X)\, \geq\, 0\, .
\end{equation}
Recall that $c_1(TX)\,\geq\, 0$ and $X$ admits a K\"ahler--Einstein metric. So,
\eqref{e3} contradicts the assumption that $c_1(TX)\,\not=\, 0$. Therefore,
we conclude that
\begin{equation}\label{e5}
c_1(TX)\,=\, 0\, .
\end{equation}
Consequently, $\omega$ is Ricci--flat, in particular, $X$ is a Calabi--Yau manifold.
\end{proof}

A well-known theorem due to Simpson says that $E$ admits an Hermitian structure that 
satisfies the Yang--Mills--Higgs equation for $(E\, ,\theta)$ if and only if $(E\, 
,\theta)$ is polystable \cite[Thm.~1]{Si1} (see also \cite{Si2}); when $X$ is a 
compact Riemann surface and $\text{rank}(E)\,=\, 2$, this was first proved in 
\cite{Hi}.
 
The Chern connection on $E$ for $h$ will be denoted by $\nabla^h$. Let
$\widehat{\nabla}^h$ denote the connection on ${\rm End}(E)\,=\,
E\bigotimes E^*$ induced by $\nabla^h$.
The Levi--Civita connection on $\Omega_X$ associated to $\omega$ and the connection
$\widehat{\nabla}^h$ on ${\rm End}(E)$ together
produce a connection on ${\rm End}(E)\bigotimes
\Omega_X$. This connection on ${\rm End}(E)\bigotimes\Omega_X$ will be denoted
by $\nabla^{\omega,h}$.

\begin{proposition}\label{prop1}
Assume that the Hermitian structure $h$ satisfies the Yang--Mills--Higgs
equation in \eqref{ymh} for $(E\, ,\theta)$.
Then the section $\theta$ of ${\rm End}(E)\bigotimes\Omega_X$ is flat (meaning
covariantly constant) with respect to the connection $\nabla^{\omega,h}$ constructed
above.
\end{proposition}

\begin{proof}
The Hermitian structure $h$ on $E$ produces an Hermitian structure on 
${\rm End}(E)$, which will be denoted by $\widehat{h}$. The connection
$\widehat{\nabla}^h$ on ${\rm End}(E)$ defined earlier is in fact the Chern connection
for $\widehat{h}$. The K\"ahler form $\omega$ and the Hermitian structure
$\widehat{h}$ together produce an Hermitian
structure on ${\rm End}(E)\bigotimes\Omega_X$. This Hermitian structure on
${\rm End}(E)\bigotimes\Omega_X$ will be denoted by $h^\omega$. We note that
the connection $\nabla^{\omega,h}$ in the statement of the proposition is the Chern
connection for $h^\omega$.

Since $\omega$ is K\"ahler--Einstein, the Hermitian structure on $\Omega_X$ induced
by $\omega$ satisfies the Yang--Mills--Higgs equation for the Higgs vector
bundle $(\Omega_X\, ,0)$. As $h$ satisfies the Yang--Mills--Higgs equation
for $(E\, ,\theta)$, this implies that $h^\omega$ satisfies the
Yang--Mills--Higgs equation for the Higgs vector bundle $({\rm End}(E)\bigotimes\Omega_X\, ,
\theta')$ constructed in \eqref{tp}. In particular, the Higgs vector
bundle $({\rm End}(E)\bigotimes
\Omega_X\, , \theta')$ is polystable.  
The Proposition is obvious if $\theta\,=\,0$. Assume that $\theta\,\not=\, 0$;
then $\varphi$ defined in \eqref{e2} is nonzero.

Since $c_1(\Omega_X)\,=\, 0$, the inequality in \eqref{e4} is an equality. Now from
\cite[Prop.~3.3]{Si1} it follows immediately that
\begin{itemize}
\item $\varphi({\mathcal O}_X)$ in \eqref{e2} is a subbundle of ${\rm End}(E)$,

\item the orthogonal complement $\varphi({\mathcal O}_X)^\perp\,\subset\,
{\rm End}(E)\bigotimes \Omega_X$ of $\varphi({\mathcal O}_X)$ with respect to the
Yang--Mills--Higgs Hermitian structure $h^\omega$ is preserved by $\theta'$, and

\item $(\varphi({\mathcal O}_X)^\perp\, , \theta'\vert_{\varphi({\mathcal O}_X)^\perp})$
is polystable with
$$
\frac{{\rm degree}(\varphi({\mathcal O}_X)^\perp)}{{\rm rank}
(\varphi({\mathcal O}_X)^\perp)}\,=\,\frac{{\rm degree}({\rm End}(E)
\otimes\Omega_X)}{{\rm rank}({\rm End}(E)\otimes\Omega_X)}\,=\, 0\, .
$$
\end{itemize}
We note that
\cite[Prop.~3.3]{Si1} also says that the Hermitian structure on
the image of $\varphi$ induced by $h^\omega$ satisfies the Yang--Mills--Higgs
equation for the Higgs vector bundle $(\varphi({\mathcal O}_X)\, ,0)$.
Since the above orthogonal complement $\varphi({\mathcal O}_X)^\perp\,\subset\, {\rm 
End}(E)\bigotimes\Omega_X$ is a holomorphic subbundle,
\begin{itemize}
\item the connection 
$\nabla^{\omega,h}$ preserves $\varphi({\mathcal O}_X)$,

\item and the 
connection on $\varphi({\mathcal O}_X)$ obtained by restricting 
$\nabla^{\omega,h}$ coincides with the Chern connection for the 
Hermitian structure $h^\omega\vert_{\varphi({\mathcal O}_X)}$.
\end{itemize}
Also, recall that $h^\omega\vert_{\varphi({\mathcal O}_X)}$ satisfies the
Yang--Mills--Higgs equation for the Higgs vector bundle $(\varphi({\mathcal O}_X)\, ,0)$.
These together imply that 
all holomorphic sections of $\varphi({\mathcal O}_X)$ over $X$ are flat with respect
to the Yang--Mills--Higgs connection $\nabla^{\omega,h}$
on ${\rm End}(E)\bigotimes\Omega_X$. In particular, the section
$\theta$ is flat with respect to $\nabla^{\omega,h}$.
\end{proof}

\subsection{Decomposition of a Higgs field}

In view of Proposition \ref{cor1}, henceforth we assume that $c_1(TX)\,=\, 0$.
Therefore, the K\"ahler--Einstein form $\omega$ is Ricci--flat.
For any point $x\, \in\, X$, the fiber of the vector bundle $\Omega_X$ over $x$
will be denoted by $\Omega_{X,x}$.

Let $(E\, ,\theta)$ be a polystable Higgs vector bundle on $X$. For any point
$x\, \in\, X$, we have a homomorphism
\begin{equation}\label{eta}
\eta_x\, :\, T_xX\, \longrightarrow\, \text{End}(E_x)\, , ~\  \eta_x(v)\,=\,
i_v(\theta(x))\, ,
\end{equation}
where $i_v\, :\, \Omega_{X,x}\, \longrightarrow\, \mathbb C$,
$z\, \longmapsto\, z(v)$, is the contraction of forms by the tangent vector $v$.

\begin{lemma}\label{lem1}
For any two points $x$ and $y$ of $X$, there are isomorphisms
$$
\alpha\, :\, T_xX\, \longrightarrow\, T_yX~ \ \text{ and }~\
\beta\, :\, E_x\, \longrightarrow\, E_y
$$
such that $\beta(\eta_x(v)(u))\,=\, (\eta_y(\alpha(v)))(\beta(u))$ for
all $v\, \in\, T_xX$ and $u\, \in\, E_x$.
\end{lemma}

\begin{proof}
Let $h$ be an Hermitian structure on $E$ satisfying the Yang--Mills--Higgs
equation for $(E\, ,\theta)$. As before, the Chern connection on $E$
associated to $h$ will be denoted by $\nabla^h$.

Fix a $C^\infty$ path $\gamma\, :\, [0\, ,1]\, \longrightarrow\, X$ such that
$\gamma(0)\,=\, x$ and $\gamma(1)\,=\, y$.
Take $\alpha$ to be the parallel transport of $T_xX$ along $\gamma$
for the Levi--Civita connection associated to $\omega$. Take $\beta$ to be
the parallel transport of $E_x$ along $\gamma$ for the above connection
$\nabla^h$. Using Proposition \ref{prop1} it is straightforward to deduce that
$$
\beta(\eta_x(v)(u))\,=\, (\eta_y(\alpha(v)))(\beta(u))
$$
for all $v\, \in\, T_xX$ and $u\, \in\, E_x$.
\end{proof}

{}From \eqref{e1} it follows immediately that for any $v_1\, , v_2\, \in\, T_xX$,
we have
$$
\eta_x(v_1)\circ \eta_x(v_2)\,=\, \eta_x(v_2)\circ \eta_x(v_1)\, ,
$$
where $\eta_x$ is constructed in \eqref{eta}. In view of this commutativity, there
is a generalized eigenspace decomposition of $E_x$ for $\{\eta_x(v)\}_{v\in T_xX}$.
More precisely, we have
distinct elements $u^x_1\, ,\cdots\, , u^x_m\, \in\, \Omega_{X,x}$ and a decomposition
\begin{equation}\label{eta2}
E_x\,=\, \bigoplus_{i=1}^m E^i_x
\end{equation}
such that
\begin{itemize}
\item for all $v\,\in\, T_x$ and $1\, \leq\, i\, \leq\, m$,
\begin{equation}\label{equation1}
\eta_x(v)(E^i_x)\, \subseteq\, E^i_x\, ,
\end{equation}

\item the endomorphism of $E^i_x$
\begin{equation}\label{equation2}
\eta_x(v)\vert_{E^i_x} - u^x_i(v)\cdot {\rm Id}_{E^i_x}
\end{equation}
is nilpotent.
\end{itemize}
Therefore, these elements $\{u^x_i\}_{i=1}^m$ are the joint generalized eigenvalues of
$\{\eta_x(v)\}_{v\in T_xX}$. Note however that there is no ordering of the elements
$\{u^x_i\}_{i=1}^m$. From Lemma \ref{lem1} it follows immediately that the integer $m$
is independent of $x$.

Let $Y'$ denote the space of all pairs of the form $(x\, , \epsilon)$, where $x\, \in\, X$ and
$$
\epsilon\, :\, \{1\, , \cdots\, , m\}\, \longrightarrow\, \{u^x_i\}_{i=1}^m
$$
is a bijection. Clearly, $Y'$ is an \'etale Galois cover of $X$ with the permutations
of $\{1\, , \cdots\, , m\}$ as the Galois group. We note
that $Y'$ need not be connected. Fix a connected component $Y\,\subset\, Y'$. Let
\begin{equation}\label{vp}
\varpi\, :\, Y\, \longrightarrow\, X\, ,\, ~ \  (x\, ,\epsilon)\,\longmapsto\, x
\end{equation}
be the projection. So $\varpi$ is an \'etale Galois covering map.

For any $y\,=\, (x\, ,\epsilon)\, \in\, Y$, and any $i\, \in\, \{1\, ,\cdots\, ,m\}$,
the element $\epsilon(i)\,\in\, \{u^x_i\}_{i=1}^m$ will be denoted by
$\widehat{u}^{\varpi(y)}_i$.

Therefore, from \eqref{eta2} we have a decomposition
\begin{equation}\label{eta3}
\varpi^*E \,=\, \bigoplus_{i=1}^m F_i\, ,
\end{equation}
where the subspace $(F_i)_y\, \subset\, (\varpi^*E)_y\,=\,
E_{\varpi(y)}$, $y\, \in\, Y$, is the
subspace of $E_{\varpi(y)}$ which is the generalized simultaneous eigenspace of
$\{\eta_x(v)\}_{v\in T_{\varpi(y)}X}$ for the eigenvalue $\widehat{u}^{\varpi(y)}_i(v)$
(the element $\widehat{u}^{\varpi(y)}_i$ is defined above).

Clearly, \eqref{eta3} is a holomorphic decomposition of the holomorphic vector
bundle $\varpi^*E$. Consider the Higgs field $\varpi^*\theta\,\in\,
H^0(Y,\, \text{End}(\varpi^* E)\bigotimes\Omega_Y)$ on $\varpi^* E$, where
$\Omega_Y\,=\, \varpi^* \Omega_X$ is the holomorphic cotangent bundle of $Y$. From
\eqref{equation1} it follows immediately that
\begin{equation}\label{et1}
(\varpi^*\theta)(F_i)\, \subseteq\, F_i\otimes \Omega_Y\, .
\end{equation}
Let
\begin{equation}\label{et2}
\theta_i\, :=\, (\varpi^*\theta)\vert_{F_i}
\end{equation}
be the Higgs field on $F_i$ obtained by restricting $\varpi^*\theta$.

Equip $Y$ with the pulled back K\"ahler form $\varpi^*\omega$. Consider the Hermitian
structure $\varpi^* h$ on $\varpi^* E$, where $h$, as before, is a Hermitian
structure on $E$ satisfying the Yang--Mills--Higgs
equation for $(E\, ,\theta)$. It is straightforward to check that $\varpi^* h$ satisfies
the Yang--Mills--Higgs equation for $(\varpi^*E\, , \varpi^*\theta)$. In particular,
$(\varpi^*E\, , \varpi^*\theta)$ is polystable. The restriction of
$\varpi^* h$ to the subbundle $F_i$ in \eqref{eta3} will be denoted by $h_i$. Since
$$
(\varpi^*E\, , \varpi^*\theta)\,=\, \bigoplus_{i=1}^m (F_i\, , \theta_i)\, ,
$$
where $\theta_i$ is constructed in \eqref{et2}, and $\varpi^*h$ satisfies
the Yang--Mills--Higgs equation for $(\varpi^*E\, , \varpi^*\theta)$,
it follows that $h_i$
satisfies the Yang--Mills--Higgs equation for $(F_i\, , \theta_i)$
\cite[p. 878, Theorem 1]{Si1}. Consequently, $(F_i\, , \theta_i)$ is polystable.
We note that the polystability of $(F_i\, , \theta_i)$ also follows form the
fact that $(F_i\, , \theta_i)$ is a direct summand of the polystable Higgs
vector bundle $(\varpi^*E\, , \varpi^*\theta)$.

Let
\begin{equation}\label{tr}
\text{tr}(\theta_i)\, \in\, H^0(Y,\, \Omega_Y)
\end{equation}
be the trace of $\theta_i$. Let $r_i$ be the rank of the vector bundle $F_i$. Define
\begin{equation}\label{t1}
\widetilde{\theta}_i\,:=\, \theta_i - \frac{1}{r_i}\text{Id}_{F_i}\otimes
\text{tr}(\theta_i)\,\in\, H^0(Y,\, {\rm End}(F_i)\otimes\Omega_Y)\, .
\end{equation}
We note that $\widetilde{\theta}_i$ is also a Higgs field on $F_i$.

\begin{corollary}\label{cor2}
The section $\theta_i\, \in\, H^0(Y,\, {\rm End}(F_i)\bigotimes\Omega_Y)$
in \eqref{et2} is flat
with respect to the connection on ${\rm End}(F_i)\bigotimes\Omega_Y$ constructed from
$h_i$ and $\varpi^*\omega$. Similarly, $\widetilde{\theta}_i$ in \eqref{t1} is flat
with respect to this connection on ${\rm End}(F_i)\bigotimes\Omega_Y$.
\end{corollary}

\begin{proof}
We noted earlier that $h_i$ satisfies the Yang--Mills--Higgs equation for
$(F_i\, ,\theta_i)$. From this it follows that $h_i$ also satisfies the Yang--Mills--Higgs
equation for $(F_i\, , \widetilde{\theta}_i)$. Therefore, substitutions of $(F_i\, , \theta_i\, ,
h_i)$ and $(F_i\, , \widetilde{\theta}_i\, ,h_i)$ in place of
$(E\, ,\theta\, ,h)$ in Proposition \ref{prop1} yield the result.
\end{proof}

\begin{proposition}\label{prop2}
The Higgs field $\widetilde{\theta}_i$ on $F_i$ in \eqref{t1} vanishes identically.
\end{proposition}

\begin{proof}
Since the endomorphism in \eqref{equation2} is nilpotent, it follows that
$$
\widetilde{\theta}_i(y)(v)\, \in\, {\rm End}(\varpi^* E_y)\,=\,
\varpi^*{\rm End}(E_y)\,=\, {\rm End}(E_{\varpi(y)})
$$
is nilpotent for all $y\, \in\, Y$ and $v\, \in\, T_yY$. Consider the homomorphism
\begin{equation}\label{hh}
\widetilde{\widetilde{\theta}}_i\, :\, F_i\, \longrightarrow\, F_i\otimes
\Omega_Y\, , ~\ z\,\longmapsto\, \widetilde{\theta}_i(y)(z)\, ~ \ \forall\
z\, \in\, (F_i)_y\, .
\end{equation}
Let
\begin{equation}\label{cV}
{\mathcal V}_i\, :=\, \text{kernel}(\widetilde{\widetilde{\theta}}_i)\, \subset\, F_i
\end{equation}
be the kernel of it. From Corollary \ref{cor2} it follows that the subsheaf
${\mathcal V}_i\, \subset\, F_i$ is a subbundle. We also note that
${\mathcal V}_i$ is of positive rank.

Let
$$
{\widetilde{\theta}}^f_i\,=\, \widetilde{\theta}_i\otimes\text{Id}_{\Omega_Y}
$$
be the Higgs field on $F_i\bigotimes \Omega_Y$. Since $\varpi^*\omega$ is
K\"ahler--Einstein, and $h_i$ satisfies the Yang--Mills--Higgs equation for
$(F_i\, , \widetilde{\theta}_i)$, the Hermitian structure on $F_i\bigotimes \Omega_Y$
induced by the combination of $h_i$ and $\varpi^*\omega$ satisfies the
Yang--Mills--Higgs equation for 
$(F_i\bigotimes \Omega_Y\, , \widetilde{\theta}^f_i)$. In particular,
$(F_i\bigotimes \Omega_Y\, , \widetilde{\theta}^f_i)$ is polystable.

Note that
\begin{equation}\label{di}
\frac{{\rm degree}(F_i\otimes \Omega_Y)}{{\rm rank}(F_i\otimes \Omega_Y)}
\,=\, \frac{{\rm degree}(F_i)}{{\rm rank}(F_i)}+
\frac{{\rm degree}(\Omega_Y)}{{\rm rank}(\Omega_Y)}\,=\,
\frac{{\rm degree}(F_i)}{{\rm rank}(F_i)}\, ;
\end{equation}
the last equality follows from the fact that
$c_1(\Omega_Y)\,=\, 0$. The homomorphism $\widetilde{\widetilde{\theta}}_i$
in \eqref{hh} is compatible with the Higgs fields 
$\widetilde{\theta}_i$ and ${\widetilde{\theta}}^f_i$ on $F_i$ and
$F_i\bigotimes \Omega_Y$ respectively, meaning $\widetilde{\theta}_i\circ
\widetilde{\widetilde{\theta}}_i\,=\, \widetilde{\widetilde{\theta}}_i\circ
\widetilde{\theta}_i$. From the definition of ${\mathcal V}_i$ in \eqref{cV}
it follows immediately that $\widetilde{\theta}_i\vert_{{\mathcal
V}_i}\,=\, 0$. Hence $({\mathcal V}_i\, ,0)$ is
a Higgs subbundle of $(F_i\, ,\widetilde{\theta}_i)$. Since both
$(F_i\, , \widetilde{\theta}_i)$ and $(F_i\bigotimes \Omega_Y\, , \widetilde{\theta}^f_i)$
are semistable of same slope (see \eqref{di}), we conclude that
$({\mathcal V}_i\, ,0)$ is a Higgs subbundle of
$(F_i\, , \widetilde{\theta}_i)$ of same slope (same as that of
$F_i$). Now, as $(F_i\, , \widetilde{\theta}_i)$ is polystable, the Higgs subbundle
$({\mathcal V}_i\, ,0)$ of same slope has a direct summand.

Let $(W_i\, , {\theta}^c_i)\, \subset\, (F_i\, , \widetilde{\theta}_i)$
be a direct summand of $({\mathcal V}_i\, ,0)$. If $W_i\,=\,0$, then the proof
is complete. So assume that $W_i\,\not=\, 0$.

Substituting $(W_i\, , {\theta}^c_i)$ in place of $(F_i\, , \widetilde{\theta}_i)$
in the above argument and iterating the argument, we conclude that
$\widetilde{\theta}_i\,=\, 0$.
\end{proof}

\begin{corollary}\label{cor3}
Let $X$ be a compact $1$--connected Calabi--Yau manifold. If $(E\, ,\theta)$ is
a polystable Higgs vector bundle on $X$, then $\theta\,=\, 0$.
\end{corollary}

\begin{proof}
Since $X$ is simply connected, it follows that $\varpi$ in \eqref{vp} is an
isomorphism. We have $H^0(X,\, \Omega_X)\,=\, 0$, because $b_1(X)\,=\, 0$
and $\dim H^0(X,\, \Omega_X)\,=\, b_1(X)/2$. Therefore,
$\text{tr}(\theta_i)$ in \eqref{tr} vanishes identically, and hence
$\widetilde{\theta}_i$ in \eqref{t1} is $\theta_i$ itself. Now Proposition
\ref{prop2} completes the proof.
\end{proof}

\bigskip
\section{Independence of Yang--Mills--Higgs Hermitian structure}

As before, $X$ is a compact connected K\"ahler manifold with $c_1(TX)\,=\, 0$, and
$\omega$ is a Ricci--flat K\"ahler form on $X$. Let $(E\, ,\theta)$ be a polystable 
Higgs vector bundle on $X$. Let $h$ be a Hermitian structure on $E$ satisfying
the Yang--Mills--Higgs equation for $(E\, ,\theta)$. We will continue to use the
set-up of Section \ref{sec2}.

\begin{lemma}\label{lem2}
The decomposition in \eqref{eta3} is orthogonal with respect to the pulled back
Hermitian structure $\varpi^*h$ on $\varpi^*E$.
\end{lemma}

\begin{proof}
The decomposition in \eqref{eta3} gives a decomposition of the Higgs vector bundle
$(\varpi^*E\, ,\varpi^*\theta)$
$$
(\varpi^*E\, ,\varpi^*\theta)\,=\, \bigoplus_{i=1}^m (F_i\, ,\theta_i)\, ,
$$
where $\theta_i$ are constructed in \eqref{et2}. Recall
that $(\varpi^*E\, ,\varpi^*\theta)$ and all
$(F_i\, ,\theta_i)$ are polystable. If $\widetilde{h}_i$, $1\,\leq\, i\,\leq\, m$,
is a  Hermitian structure on $F_i$ satisfying the Yang--Mills--Higgs
equation for $(F_i\, ,\theta_i)$, then
the Hermitian structure $\bigoplus_{i=1}^m \widetilde{h}_i$ on $\varpi^*E$,
constructed using the decomposition in \eqref{eta3}, clearly satisfies
the Yang--Mills--Higgs equation for $(\varpi^*E\, ,\varpi^*\theta)$.

Any two Hermitian structures on $\varpi^*E$ that satisfy the
Yang--Mills--Higgs equation for $(\varpi^*E\, , \varpi^*\theta)$, differ
by a holomorphic automorphism of the Higgs vector bundle
$(\varpi^*E\, ,\varpi^*\theta)$ \cite[p. 878, Theorem 1]{Si1}. In particular, there
is a holomorphic automorphism
$$
T\, :\, \varpi^*E\,\longrightarrow\, \varpi^*E
$$
such that $(T\otimes\text{Id}_{\Omega_Y})\circ (\varpi^*\theta)\,=\,
(\varpi^*\theta)\circ T$, and
\begin{equation}\label{eq3}
\bigoplus_{i=1}^m \widetilde{h}_i (a\, ,b)\,=\, \varpi^* h(T(a)\, , T(b))\, .
\end{equation}
Therefore, the lemma follows once it is shown that
any holomorphic automorphism of the Higgs vector bundle $(\varpi^*E\, ,\varpi^*\theta)$
preserves the decomposition in \eqref{eta3}. Note that the decomposition in \eqref{eta3}
is orthogonal for the above Hermitian structure
$\bigoplus_{i=1}^m \widetilde{h}_i$ on $\varpi^*E$. If the above automorphism $T$
preserves the decomposition in \eqref{eta3}, then from \eqref{eq3} it follows immediately
that the decomposition in \eqref{eta3} is orthogonal with respect to $\varpi^* h$.

{}From the construction of the decomposition in \eqref{eta3} it follows that the
$m$ sections
$$
\frac{1}{r_1}\text{tr}(\theta_1)\, ,\cdots\, , 
\frac{1}{r_m}\text{tr}(\theta_m)\, \in\, H^0(Y,\, \Omega_Y)
$$
in \eqref{tr} and \eqref{t1} are distinct; as mentioned just before
\eqref{eta2}, the elements $\{u^x_i\}_{i=1}^m$ are all distinct. Indeed, \eqref{eta3}
is the generalized eigenspace decomposition for $\varpi^*\theta$, and
$\frac{1}{r_1}\text{tr}(\theta_1)\, ,\cdots\, ,
\frac{1}{r_m}\text{tr}(\theta_m)$ are the eigenvalues. It now follows that
any automorphism of the Higgs vector bundle $(\varpi^*E\, ,\varpi^*\theta)$ preserves
the decomposition in \eqref{eta3}. As observed earlier, this completes the proof.
\end{proof}

\begin{lemma}\label{lem3}
The section
$$
\theta\bigwedge\theta^*\, \in\, C^\infty(X,\, {\rm End}(E)\otimes \Omega^{1,1}_X)
$$
(see \eqref{ymh}) vanishes identically.
\end{lemma}

\begin{proof}
Consider $\theta_i$ defined in \eqref{et2}. From Proposition \ref{prop2} it follows
immediately that
\begin{equation}\label{en}
\widetilde{\theta}_i\bigwedge\widetilde{\theta}^*_i\,=\,0\, .
\end{equation}
Since the decomposition in \eqref{eta3} is orthogonal by Lemma \ref{lem2}, from
\eqref{en} and \eqref{t1} we conclude that
$$
(\varpi^*\theta)\bigwedge(\varpi^*\theta^*)\,=\,0\, .
$$
This implies that $\theta\wedge\theta^*\,=\, 0$.
\end{proof}

\begin{theorem}\label{thm1}
Let $(E\, ,\theta)$ be a polystable Higgs vector bundle on $X$ equipped with a
Yang--Mills--Higgs structure $h$. Then $h$ also satisfies the Yang--Mills--Higgs equation
for the Higgs vector bundle $(E\, ,0)$.
\end{theorem}

\begin{proof}
In view of Lemma \ref{lem3}, this follows immediately from \eqref{ymh}.
\end{proof}

\begin{remark}\label{rem1}
{\rm It should be clarified that the converse of Theorem \ref{thm1} is not valid.
In other words, if $h$ is an Hermitian structure on $E$ satisfying the
Yang--Mills--Higgs equation for $(E\, ,0)$, then $h$ need not
satisfy the Yang--Mills--Higgs equation for $(E\, ,\theta)$. The reason for it is that
the automorphism group of $(E\, ,0)$ is in general bigger than the
automorphism group of $(E\, ,\theta)$. To give an example, take $X$ to be a complex
elliptic curve equipped with a flat metric. Take $E$ to be the trivial vector
bundle ${\mathcal O}^{\oplus 2}_X$ on $X$ of rank two. Let $\theta$ be the Higgs field on
${\mathcal O}^{\oplus 2}_X$ given by the matrix
$$
A\,:=\, 
\begin{pmatrix}
1 & 1\\
0 & 2
\end{pmatrix}\, ;
$$
fixing a trivialization of $\Omega_X$, we identify the Higgs fields on
${\mathcal O}^{\oplus 2}_X$ with the $2\times 2$ complex matrices. This
Higgs vector bundle $(E\, , \theta)$ is polystable because the matrix $A$ is
semisimple. The Hermitian structure on ${\mathcal O}^{\oplus 2}_X$ given by the standard
inner product on ${\mathbb C}^2$ satisfies the Yang--Mills--Higgs equation for $(E\, ,
0)$, but this Hermitian structure does not satisfy Yang--Mills--Higgs equation for
$(E\, , \theta)$ (because $AA^*\, \not=\, A^*A$).}
\end{remark}

\begin{remark}\label{rem2}
{\rm Let $(E\, ,\theta)$ be a polystable Higgs vector bundle on $X$. From
Theorem \ref{thm1} we know that the Higgs vector bundle $(E\, ,0)$ is polystable. Fix
a Hermitian structure $h_0$ on $E$ satisfying the Yang--Mills--Higgs equation for
$(E\, ,0)$. Any other Hermitian structure on $E$ that satisfies the Yang--Mills--Higgs
equation for $(E\, ,0)$ differs from $h_0$ by a holomorphic automorphism of $E$. Take
a holomorphic automorphism $T$ of $E$ such that the Hermitian structure $h\, :=\, T^*h_0$
on $E$ has the following property:
$$
\theta\bigwedge \theta^{*_h}\,=\, 0\, ,
$$
where $\theta^{*_h}$ is the adjoint of $\theta$ constructed using $h$. From
Lemma \ref{lem3} and Theorem \ref{thm1} it follows that such an automorphism $T$ exists.
The above Hermitian structure $h$ satisfies the Yang--Mills--Higgs equation for
$(E\, ,\theta)$.}
\end{remark}

\bigskip
\section{Polystable principal Higgs $G$--bundles}\label{se4}

Let $G$ be a connected reductive affine algebraic group defined over $\mathbb C$.
The Lie algebra of $G$ will be denoted by $\mathfrak g$.
As before, $X$ is a compact connected K\"ahler manifold 
equipped with a Ricci--flat K\"ahler form $\omega$.
Let $E_G\, \longrightarrow\, X$ be a holomorphic principal $G$--bundle. Its adjoint
vector bundle $E_G\times^G\mathfrak g$ will be denoted by $\text{ad}(E_G)$.
A Higgs field on $E_G$ is a holomorphic section
$$
\theta\, \in\, H^0(X,\, \text{ad}(E_G)\otimes\Omega_X)
$$
such that the section $\theta\bigwedge\theta$ of $\text{ad}(E_G)\bigotimes\Omega^2_X$
vanishes identically. A Higgs $G$--bundle on $X$ is a pair of the form $(E_G\, ,\theta)$,
where $E_G$ is a holomorphic principal $G$--bundle on $X$, and $\theta$ is a
Higgs field on $E_G$.

Fix a maximal compact subgroup
$$
K_G\, \subset\, G\, .
$$
A Hermitian structure on a holomorphic principal $G$--bundle $E_G$ on $X$ is a
$C^\infty$ reduction of structure group of $E_G$
$$
E_{K_G}\, \subset\, E_G
$$
to the subgroup $K_G$. There is a unique $C^\infty$ connection $\nabla$ on the principal
$K_G$--bundle $E_{K_G}$ such that the connection on $E_G$ induced by $\nabla$ is
compatible with the holomorphic structure of $E_G$ \cite[p. 191--192, Proposition
5]{At}. Using the decomposition ${\mathfrak g}\,=\, \text{Lie}(K)\oplus\mathfrak p$,
given any Higgs field $\theta$ on $E_G$, we have
$$
\theta^*\, \in\, C^\infty(X;\, \text{ad}(E_G)\otimes \Omega^{0,1}_X)\, .
$$

Let $(E_G\, ,\theta)$ be a Higgs $G$--bundle on $X$.
The center of the Lie algebra $\mathfrak g$ will be denoted by $z(\mathfrak g)$.
Since the adjoint action of $G$ on $z(\mathfrak g)$ is trivial, we have an injective
homomorphism
\begin{equation}\label{psi}
\psi\, :\, X\times z(\mathfrak g)\, \hookrightarrow\, \text{ad}(E_G)
\end{equation}
{}from the trivial vector bundle with fiber $z(\mathfrak g)$.
This homomorphism $\psi$ produces an injective homomorphism
$$
\widehat{\psi}\, :\, z(\mathfrak g)\, \hookrightarrow\, H^0(X,\, \text{ad}(E_G))\, .
$$

A Hermitian structure $E_{K_G}\, \subset\, E_G$ is said to satisfy the Yang--Mills--Higgs
equation for $(E_G\, ,\theta)$ if there is an element $c\, \in\, z(\mathfrak g)$ such
that
$$
\Lambda_\omega({\mathcal K}(\nabla)+\theta\bigwedge\theta^*) \,=\, \widehat{\psi}(c)\, ,
$$
where ${\mathcal K}(\nabla)$ is the curvature of the connection $\nabla$ associated to
the reduction $E_{K_G}$, and $\theta^*$ is defined above.

It is known that $(E_G\, ,\theta)$ admits a Yang--Mills--Higgs Hermitian structure
if and only if $(E_G\, ,\theta)$ is polystable \cite{Si2}, \cite[p. 554, Theorem
4.6]{BS}. (See \cite{BS} for the definition of a polystable Higgs $G$--bundle.)

\begin{lemma}\label{lem4}
Let $(E_G\, ,\theta)$ be a Higgs $G$--bundle on $X$ equipped with an Hermitian
structure $E_{K_G}\, \subset\, E_G$ that satisfies the Yang--Mills--Higgs
equation for $(E_G\, ,\theta)$. Then
$$
\theta\bigwedge\theta^*\,=\, 0\, .
$$
\end{lemma}

\begin{proof}
This follows by applying Lemma \ref{lem3} to the Higgs vector bundle associated
to $(E_G\, ,\theta)$ for the adjoint action of $G$ on $\mathfrak g$. Consider the
adjoint Higgs vector bundle $(\text{ad}(E_G)\, ,\text{ad}(\theta))$. The reduction
$E_{K_G}$ produces a Hermitian structure on the vector bundle $\text{ad}(E_G)$ that satisfies
the Yang--Mills--Higgs equation for $(\text{ad}(E_G)\, ,\text{ad}(\theta))$. Now
Lemma \ref{lem3} says that
$$
\text{ad}(\theta)\bigwedge\text{ad}(\theta)^*\,=\, 0\, .
$$
This immediately implies that the $C^\infty$ section $\theta\bigwedge\theta^*$ of
$\text{ad}(E_G)\bigotimes\Omega^{1,1}_X$ is actually a section of
$\psi(z(\mathfrak g))\bigotimes \Omega^{1,1}_X$, where $\psi$ is the homomorphism
in \eqref{psi}.

Take any holomorphic character $\chi\, :\, G\, \longrightarrow\, {\mathbb C}^*$. Let
$$
L^\chi\, :=\, E_G\times^\chi \mathbb C\, \longrightarrow\, X
$$
be the holomorphic line bundle associated to $E_G$ for $\chi$. The Higgs field $\theta$
defines a Higgs field on $L^\chi$ using the homomorphism of Lie algebras
\begin{equation}\label{eq4}
d\chi\, :\, {\mathfrak g}\,\longrightarrow\, \mathbb C
\end{equation}
associated to $\chi$; this Higgs field on $L^\chi$ will be denoted by $\theta^\chi$. Since
$L^\chi$ is a line bundle, we have $\theta^\chi\bigwedge(\theta^\chi)^*\,=\, 0$
(Lemma \ref{lem3} is not needed for this). As
$\theta\bigwedge\theta^*$ is a section of
$\psi(z(\mathfrak g))\bigotimes \Omega^{1,1}_X$, from this it can be deduced that
$\theta\bigwedge\theta^*\,=\, 0$. Indeed, given any nonzero element $v\, \in\,
z({\mathfrak g})$, there is a holomorphic character
$$
\chi\, :\, G\, \longrightarrow\, {\mathbb C}^*
$$
such that $d\chi(v)\, \not=\, 0$ (defined in \eqref{eq4}).
\end{proof}

\begin{theorem}\label{thm2}
Let $(E_G\, ,\theta)$ be a polystable Higgs $G$--bundle on $X$, and let $E_{K_G}\,
\subset\, E_G$ be an Hermitian structure that satisfies the Yang--Mills--Higgs
equation for $(E_G\, ,\theta)$. Then the Hermitian
structure $E_{K_G}\, \subset\, E_G$ also satisfies the Yang--Mills--Higgs
equation for $(E_G\, ,0)$. 
\end{theorem}

\begin{proof}
In view of the Yang--Mills--Higgs equation for $(E_G\, ,\theta)$,
this follows immediately from Lemma \ref{lem4}.
\end{proof}

\section*{Acknowledgements}

We thank the referee for pointing out several misprints.
I.B.\ is supported by a J.\ C.\ Bose Fellowship.
B.G.O.\ is supported by Project {\sc  ID PRY 6777} of the Pontificia Universidad Javeriana, Bogot\'a.
U.B.\ is supported by {\sc prin} ``Geometry of Algebraic Varieties'' and  {\sc gnsaga-indam}.
A. L. is supported by the {\sc fapesp} post-doctoral grant number 2013/20617-2.


\end{document}